\RequirePackage{fix-cm}
\documentclass[smallcondensed]{svjour3}       % onecolumn (second format)
\smartqed  % flush right qed marks, e.g. at end of proof
\usepackage{amsmath} 
\usepackage{verbatim}
\usepackage{amssymb, float,empheq}
\usepackage{color}
\usepackage{graphicx,epsfig}
\graphicspath{{figures/}}
\usepackage[applemac]{inputenc}
\usepackage{graphicx}
\def\beq{\begin{equation}}
\def\eeq{\end{equation}}
\def\baq{\begin{eqnarray}}
\def\eaq{\end{eqnarray}}
\def\baqn{\begin{eqnarray*}}
\def\eaqn{\end{eqnarray*}}

\newcommand{\R}{\mathbb{R}}

\newcommand{\ball}{\mathbb{B}}

\begin{document}

\title { R-Continuity with Applications to Convergence Analysis of Tikhonov Regularization and DC Programming
}


\author{     B. K. Le 
           %etc.
}

\institute{
B. K. Le \at Optimization Research Group, Faculty of Mathematics and Statistics, Ton Duc Thang University, Ho Chi Minh City, Vietnam
              (lebakhiet@tdtu.edu.vn)
}
\date{}

\maketitle

\abstract{In the paper, we study the convergence analysis of Tikhonov regularization in finding a zero of a maximal monotone operator using the  notion of {$R$-continuity}.  Applications to  convex minimization and DC programming are provided.}

\section{Introduction}\label{sec1}
Our aim is to provide a new way to find  $x\in H$ satisfying
\beq\label{main}
0\in Ax,
\eeq
where $H$ is a given real Hilbert space and $A: H \rightrightarrows H$ is a maximal monotone operator.
It is one of the fundamental {problems} in convex optimization which has been extensively studied (see, e. g., \cite{Alvarez,Attouch,Attouch1,Chen,Cominetti,Eckstein,Lions,Martinet70,Martinet72,Nesterov,Nesterovb,Passty,Rockafellar,Solodov,Tikhonov} and the references therein). Indeed it is known that minimizing a proper convex lower semicontinuous function $f: H \to \mathbb{R}\cup \{+\infty\}$ can be reduced to the {problem}
\beq
0\in \partial f(x).
\eeq
The proximal point algorithm \cite{Martinet70,Martinet72,Rockafellar} is one of well-known methods 
 which is given by
\beq
x_0\in H, \;\;x_{k+1}=J_{\gamma A}x_k, \;\; k=1,2,3\ldots
\eeq
for some $\gamma>0$ where $J_{\gamma A}:=(Id+\gamma A)^{-1}$ denotes the resolvent of $A$. 
Weak convergence of the proximal point algorithm follows by the firm non-expansiveness of the resolvent of $A$ and Opial's Lemma \cite{Eckstein,Rockafellar}. Strong convergence can be obtained by combining  with the Krasnoselskii--Mann iteration scheme or Halpern's procedure (see, e.g., \cite{Solodov,Nakajo,Wittmann}). Linear convergence of the algorithm is an attractive property, which was confirmed by  Rockafellar \cite{Rockafellar} when $A$ is strongly monotone or more weakly, the inverse of $A$ is Lipschitz continuous at zero. However,  the Lipschitz continuity of the inverse of $A$ at zero in Rockafellar's sense (see \cite{Rockafellar,Tao}) is quite restrictive since it requires that the inverse of $A$ at zero is singleton, i. e., the problem must have a unique solution, which is not satisfied in many applications. 

Inspired by the Rockafellar's idea, in this paper we establish a new way to solve the problem  under only the {$R$-continuity} of  the inverse of $A$ at zero in the sense of set-valued analysis. It is known that the unique solution $x_\varepsilon$ of the Tikhonov regularized problem
\beq\label{regul}
0\in (A+\varepsilon Id)x_\varepsilon,
\eeq
converges to the least-norm solution  of (\ref{main}) as $\varepsilon>0$ tends to zero \cite{Cominetti,Tikhonov}.
 However, under   the {truncated $R$-continuity} of  the inverse of $A$ at zero with its explicit continuity modulus function, we can estimate {how fast is this convergence}. It helps us to decide whether we should use the regularization. The assumption, which is automatically  satisfied when $A$ is a maximal monotone operator, has a closed connection with the notion of  metric regularity (see, e. g., \cite{Bosch,Henrion,Hoffman}).
In addition, we  show that when $A$ is not necessarily monotone, the  {$R$-continuity} of the inverse of $A$ at zero is also important  in the convergence analysis of DC programming for a large class of functions.  

The paper is organized as follows.  First we recall some definitions and useful results concerning  operators in Section 2. The convergence analysis of Tikhonov regularization  is provided in Section 3.  Applications to  convex  and DC programmings are given in  Section 4. Finally, some conclusions end the paper in Section 5.

\section{Notations and preliminaries}

Let $H$ be a given real Hilbert space with its norm $\Vert \cdot \Vert$  and scalar product $\langle \cdot,\cdot \rangle$. The closed unit ball is denoted by $\ball$.
Let be given a closed, convex set $K\subset H$. The distance and the projection from a point $s$ to $K$ are defined respectively by 
$${ d}(s,K):=\inf_{x\in K} \|s-x\|, \;\;{\rm proj}_K(s):=x \in K \;\;{\rm such \;that \;} { d}(s,K)= \|s-x\|.$$
The least-norm element of $K$ is  defined by 
$ {\rm proj}_K(0).$\\

% A mapping $f: H \to H$ is called \textit{nonexpansive} if it is Lipschitz continuous with modulus $1$. It is called \textit{firmly nonexpansive} if
%$$
%\Vert f(x)-f(y)\Vert^2\le \Vert x-y\Vert^2- \Vert (Id-f)(x)-(Id-f)(y)\Vert^2,\;\forall \;x, y\in H.
%$$
\noindent The domain, the range and the graph of a set-valued mapping $\mathcal{A}: {H}\rightrightarrows {H}$ are defined respectively by
$${\rm dom}(\mathcal{A})=\{x\in {H}:\;\mathcal{A}(x)\neq \emptyset\},\;\;{\rm rge}(\mathcal{A})=\displaystyle\bigcup_{x\in{H}}\mathcal{A}(x)\;\;$$
 and
 $$\;\;{\rm gph}(\mathcal{A})=\{(x,y): x\in{H}, y\in \mathcal{A}(x)\}.$$

\noindent {It is called \textit{monotone} provided
$$
\langle x^*-y^*,x-y \rangle \ge 0 \;\;\forall \;x, y\in H, x^*\in \mathcal{A}(x) \;{\rm and}\;y^*\in \mathcal{A}(y).
$$
In addition, it is called \textit{maximal monotone} if there is no monotone operator $\mathcal{A}'$ such that the graph of $\mathcal{A}$ is strictly contained in the graph of $\mathcal{A}'$.
We have the following classical result (see, e. g., \cite{ac,Brezis}).
\begin{proposition}\label{strwe}
Let $\mathcal{A}: {H}\rightrightarrows {H}$ be a maximal monotone operator. Then the graph of $\mathcal{A}$ is strongly-weakly closed, i. e., if  $(x_n)$ converges strongly  to some $x$,  $(y_n)$ converges weakly to some $y$  and $y_n\in \mathcal{A}(x_n)$ then $y\in \mathcal{A}(x)$.
\end{proposition}
%\noindent It is called $\mu$-\textit{strongly monotone} ($\mu>0$) provided
%$$
%\langle  \mathcal{A}x-\mathcal{A}y, x-y \rangle \ge \mu\Vert x-y \Vert^2\;\;\forall\;x, y\in H.
%$$
%It is easy to see that if $\mathcal{A}$ is $\mu$-{strongly monotone} then $\mathcal{A}^{-1}$ is single-valued $\frac{1}{\mu}$-Lipschitz continuous.\\
%
%\noindent {The \textit{resolvent} and  \textit{reflected resolvent} of $\mathcal{A}$ are defined respectively as follows
%$$
%J_\mathcal{A}:=(Id+\mathcal{A})^{-1}\;\;{\rm and} \;\; R_\mathcal{A}:= 2J_\mathcal{A}-Id,
%$$
%where $Id$ denotes the identity operator.
%}
%\begin{remark}
%\textit{The resolvents and reflected resolvents of maximal monotone operators are  nonexpansive (see, e.g., \cite{Rockafellar,Bauschke}).}
%\end{remark}

Suppose that   $0\in {\rm dom}(\mathcal{A})$. {The following concept is an extension of the Lipschitz continuity introduced by Rockafellar \cite{Rockafellar}}. 
\begin{definition}The set-valued operator \noindent $\mathcal{A}$ is  called {truncately} {$R$-continuous} at $0$ if there exist $\sigma>0$ and a non-decreasing function $\rho: \mathbb{R}^+\to \mathbb{R}^+$ satisfying $\lim_{r\to 0^+}\rho(r)=\rho(0)=0$  such that
\beq\label{rho}
\mathcal{A}(x)\cap a\ball \subset \mathcal{A}(0)+\rho(\Vert x \Vert)\ball, \forall x\in \sigma \ball,
\eeq
where $a=\inf_{y\in \mathcal{A}(0)}\Vert y \Vert.$ The  function $\rho$  is called a continuity modulus function of $\mathcal{A}$ at zero. It is called {truncately} {$R$-Lipschitz continuous} at $0$ with modulus $c>0$ if $\rho(\Vert x \Vert)=c \Vert x \Vert$.

If the left-hand side of (\ref{rho}) is replaced by $\mathcal{A}(x)$ then $\mathcal{A}$ is called  {$R$-continuous} at $0$.
Similarly if $\rho(\Vert x \Vert)=c \Vert x \Vert$ for some $c>0$ then $\mathcal{A}$ is called {$R$-Lipschitz continuous} at $0$ with modulus $c$.
\end{definition}
\begin{remark}
\textit{ i) Our  Lipschitz continuity is strictly weaker than the Lipschitz continuity in the classical Rockafellar's sense \cite{Rockafellar} since it allows $\mathcal{A}(0)$ to be set-valued. Two definitions  {coincide} if $\mathcal{A}(0)$ is singleton. The ${ Sign}$ mapping, defined by
$$
{ Sign}(x) = \left\{
\begin{array}{lll}
\; -1 & \mbox{ if } & x < 0,\\
&&\\
\; [ -1,1 ] & \mbox{ if } & x = 0, \\
&&\\
\; 1 & \mbox{ if } & x > 0,
\end{array}\right.
$$
 is Lipschitz continuous at $0$ in our sense but  not Lipschitz continuous at $0$ in the  Rockafellar's sense.} \\

\textit{ii) If $A^{-1}$ is {$R$-Lipschitz} continuous at $0$ with modulus $c$ then $(kA)^{-1}$ is {$R$-Lipschitz} continuous at $0$  with modulus $\frac{c}{\vert k \vert}$ for any $k\in \mathbb{R}$.}
\end{remark}
\begin{proposition}\label{rela}
If $\mathcal{A}:  {H}\rightrightarrows {H}$ is maximal monotone with $0\in {\rm dom}(\mathcal{A})$ then $\mathcal{A}$ is  {truncately}  {$R$-continuous} at $0$.
\end{proposition}
\begin{proof}
Let $S:=\mathcal{A}(0)$. Suppose that $\mathcal{A}$  is not  {truncately} {$R$-continuous} at $0$. Let us define the function $\rho: \mathbb{R}^+\to \mathbb{R}^+$ as follows: $\rho(0)=0$ and   if $\sigma>0$, we let

$$\rho(\sigma)=\inf\{ \delta>0: \mathcal{A}(x) \cap  a\ball \subset \mathcal{A}(0)+\delta\ball,\;\;\forall x\in \sigma \ball\}$$
where $a=\Vert \tilde{y}\Vert$ and  $\tilde{y}$ is the least-norm element of $S$. \\

It is easy to see that $\rho$ is well-defined and non-decreasing. We will prove that 
$\lim_{\sigma\to 0^+}\rho(\sigma)=\rho(0)=0$.  Since $\rho$ is non-decreasing and bounded below by $0$, the limit $\lim_{\sigma\to 0^+}\rho(\sigma)$ exists. Suppose that $\lim_{\sigma\to 0^+}\rho(\sigma)=\delta^*>0$ then there exists  $(x_n)$ such that $x_n \downarrow 0$
$$
\mathcal{A}(x_n) \cap  a\ball \not\subset \mathcal{A}(0)+\frac{\delta^*}{2}\ball.
$$
Therefore there exists  a sequence  $(y_n)$ such that  $y_n\in \mathcal{A}(x_n)\;\cap \;a\ball$  and 
\beq\label{contrac}
y_n\notin \mathcal{A}(0)+\frac{\delta^*}{2}\ball.
\eeq

Let $y^*$ be a weak limit point of $(y_n)$. The existence of $y^*$ is assured due to the fact that $y_n \in a\ball$. Note that the graph of $\mathcal{A}$ is strongly-weakly closed (Proposition \ref{strwe}), we must have $y^*\in S$. On the other hand, since the norm function is weakly continuous we deduce that
$$
\Vert y^*\Vert\le \liminf \Vert y_n \Vert \le \limsup \Vert y_n \Vert \le a. 
$$
Thus $y^*=\tilde{y}$ and $\lim_{n\to+\infty}  \Vert y_n \Vert =a=\Vert \tilde{y} \Vert$. Since $y^*$ is an arbitrarily weak limit point of $(y_n)$, this implies that $y_n$ converges weakly to $\tilde{y}$ with  $\lim_{n\to+\infty}  \Vert y_n \Vert =\Vert \tilde{y} \Vert$.  Consequently $(y_n)$ convergs  to  $\tilde{y}$. Combining with (\ref{contrac}), we obtain a contradiction and the conclusion follows.\qed
\end{proof}
{\begin{remark}
We provide not only the truncated $R$-continuity of a maximal monotone operator but also an associated modulus function $\rho$.
\end{remark}}
Using similar arguments, one has the following result.
\begin{proposition}\label{compact}
If $\mathcal{A}:  {H}\rightrightarrows {H}$ has closed graph with compact range and $0\in {\rm dom}(\mathcal{A})$ then $\mathcal{A}$ is   { $R$-continuous} at $0$.
\end{proposition}
\begin{proof}
Suppose that $\mathcal{A}$  is not { $R$-continuous} at $0$. {We define the function $\rho$ as follows: $\rho(0)=0$ and   if $\sigma>0$, we let
$$\rho(\sigma)=\inf\{ \delta>0: \mathcal{A}(x)  \subset \mathcal{A}(0)+\delta\ball,\;\;\forall x\in \sigma \ball\}.$$
 Then we must have $\lim_{\sigma\to 0^+}\rho(\sigma)=\delta^*>0$ and  similarly as in the proof of Proposition \ref{rela}}, there exist  two sequences $(x_n)$, $(y_n)$ such that $x_n \to 0$, $y_n\in \mathcal{A}(x_n)$  and 
\beq\label{closed}
y_n\notin \mathcal{A}(0)+\frac{\delta^*}{2}\ball.
\eeq
Since $\mathcal{A}$ has compact range, there exists a convergent subsequence $(y_{n_k})$. Let $y^*=\lim_{k\to +\infty}y_{n_k}$. Then $y^*\in  \mathcal{A}(0)$ because $\mathcal{A}$ has closed graph. However it is a contradiction to $(\ref{closed})$ and the conclusion follows. \qed
\end{proof}
{\begin{remark}
Note that $A^{-1}:=\mathcal{A}$ has closed graph with compact range if and only if $A$ has closed graph with compact domain. For example, $A=f+N_K$ where $f$ is a continuous function and $K$ is a convex compact set. This is the case of minimizing  $C^1$ functions on convex compact sets, as we can see in the Corollary \ref{exam} when we deal with DC programming. 
\end{remark}}
A linear mapping $B: H \to H$ is called \textit{positive semidefinite} if
$$
\langle Bx, x \rangle \ge 0,\;\; \forall \;x\in H.
$$
The next  statement is classical (see, e. g., \cite{abc}). However we provide here   an explicit  constant in the estimation. 
\begin{lemma}\label{lpeigen}
Let $B\in \mathbb{R}^{n\times n}$ be a symmetric positive semidefinite matrix. Then for all $x\in {\rm rge}(B)$, we have
$$
\langle Bx, x \rangle\ge k\Vert x \Vert^2,
$$
and thus 
\beq
\Vert Bx \Vert \ge k \Vert x \Vert,
\eeq
where $k$ is the least positive eigenvalue of $B$.
\end{lemma}
\begin{proof}
Since $B$ is symmetric positive semidefinite matrix, we can write $B=O^TDO$ where $D$ is a diagonal matrix with nonnegative eigenvalues and $O$ is {an} orthogonal matrix, i. e., $OO^T=Id$. If $x\in {\rm rge}(B)$ then $x=By=O^TDOy$ for some $y\in\mathbb{R}^n$ and $Ox=OO^TDOy=DOy\in {\rm rge}(D)$. Thus
$$
\langle Bx, x \rangle= \langle O^TDOx, x \rangle= \langle DOx, Ox\rangle\ge k \Vert Ox \Vert^2=k\Vert x \Vert^2,
$$
where $k$ is the least positive eigenvalue of $D$ and also of $B$.  \qed
\end{proof}
The following result is a consequence of Lemma \ref{lpeigen}.
\begin{lemma}\label{matr}
Let $B\in \mathbb{R}^{n\times n}$ be a symmetric matrix. Then ${\rm ker}(B)={\rm ker}(B^2)$, \\${\rm rge}(B^2)\subset {\rm rge}(B)$ and for all $x\in {\rm rge}(B^2)$, we have
\beq
\Vert Bx \Vert \ge \frac{k}{\Vert B\Vert} \Vert x \Vert,
\eeq
where $k$ is the least positive eigenvalue of $B^2$ and {\rm ker}(B) denotes the kernel of $B$.
\end{lemma}
\begin{proof}
It is easy to see that ${\rm ker}(B)={\rm ker}(B^2), {\rm rge}(B^2)\subset {\rm rge}(B)$ and $B^2$ is a symmetric positive semidefinite matrix. Using Lemma \ref{lpeigen}, for all $x\in {\rm rge}(B^2)$ one has
\beq
k \Vert x \Vert^2\le \langle B^2x, x \rangle\le  \Vert B\Vert \Vert Bx\Vert \Vert x \Vert,
\eeq
and the conclusion follows. \qed
\end{proof}
The following propositions provides some cases such that $A^{-1}$ is {$R$-Lipschitz} continuous at zero (see also \cite{Hoffman} for the Hoffman constant). They can be used to estimate the error for the linear system $Bx=C$ when $B$ is not an invertible matrix. 
\begin{proposition}\label{lipm}
Let  $A(x):=Bx-C, \forall x\in \mathbb{R}^n$ where $B\in \mathbb{R}^{n\times n}$ is a symmetric, positive semi-definite matrix, $C$ is some vector  in  $\mathbb{R}^n$. Suppose that $S:=A^{-1}(0)\neq \emptyset$. Then $A^{-1}$ is {$R$-Lipschitz} continuous at zero with modulus  $\frac{1}{k}$, where $k$ is the least positive eigenvalue of $B$.
\end{proposition}
\begin{proof}
Since $S\neq \emptyset$, one has $S={\rm ker}(B)+x_r$ where  $x_r$ is the unique element of  ${\rm rge}(B)$ such that $Bx_r=C$.
Let $y\in \mathbb{R}^n$, if $A^{-1}(y)$ is empty then the conclusion holds trivially.  If $A^{-1}(y)$ is nonempty, i. e., $y\in {\rm rge}(B)$ then  
$A^{-1}(y)={\rm ker}(B)+y_r$ 
where $y_r$ is the unique element of  ${\rm rge}(B)$ such that $By_r=C+y$. Since $\Vert y_r-x_r\Vert\le \frac{1}{k}\Vert B(y_r-x_r) \Vert = \frac{1}{k} \Vert y \Vert$  (see Lemma \ref{lpeigen}), we conclude that $A^{-1}$ is {$R$-Lipschitz} continuous at zero with modulus $\frac{1}{k}$.\qed
\end{proof}

\begin{proposition}\label{lip}
Let  $A(x):=Bx-C, \forall x\in \mathbb{R}^n$ where $B\in \mathbb{R}^{n\times n}$ is a symmetric matrix and $C$ is some vector  in  $\mathbb{R}^n$. Suppose that $S:=A^{-1}(0)\neq \emptyset$. Then $A^{-1}$ is {$R$-Lipschitz} continuous at zero with modulus $\frac{\Vert B \Vert}{k}$, where $k$ is the least positive eigenvalue of $B^2$.
\end{proposition}
\begin{proof}
Let $x_1, x_2\in S$ then $x_1-x_2\in {\rm ker}(B)={\rm ker}(B^2)$, i. e., $x_1$ and $x_2$ has the same projection onto ${\rm rge}(B^2)$, which is called $x_r$ and one has $Bx_r=C$. Thus $S=x_r+{\rm ker}(B)$. 

Next we prove that $A^{-1}$ is {$R$-Lipschitz} continuous at zero. Let $y\in \mathbb{R}^n$, if $A^{-1}(y)$ is empty then the conclusion is trivial.  If $A^{-1}(y)$ is nonempty, i. e., $y\in {\rm rge}(B)$ then  
$A^{-1}(y)={\rm ker}(B)+y_r$ 
where $y_r$ is the unique element of  ${\rm rge}(B^2)$ such that $By_r=C+y$. Since $\Vert y_r-x_r\Vert\le \frac{\Vert B \Vert}{k}\Vert B(y_r-x_r) \Vert = \frac{\Vert B \Vert}{k} \Vert y \Vert$  (see Lemma \ref{matr}), the conclusion follows.\qed
\end{proof}
{In the sequel, we recall a  general class of convex continuous differentiable functions with quadratic functional growth, that allows us to obtain linear convergence rates using first order methods \cite{Necoara}. We show that this class is  smaller than our proposed class.
\begin{definition}
\textit{Continuously differentiable function $f$ has a quadratic functional
growth on $H$ if there exists a constant $\kappa_f>0$ such that for any $x\in H$ and $x^*={\rm proj}_S(x)$, we have }
\beq\label{quad}
f(x)-f^*\ge \frac{\kappa_f}{2}\Vert x-x^*\Vert^2,
\eeq
where $S:={\rm argmin}_H f $ and $f^*=f(x^*)$.
\end{definition}
\begin{proposition}\label{propqu}
If the convex continuous differentiable function $f: H\to \R$ has a quadratic functional
growth on $H$, then $(\nabla f)^{-1}$ is $R$-Lipschitz continuous at zero.
\end{proposition}
\begin{proof}
We prove that for given $y\in H$, we have
$$
(\nabla f)^{-1}(y)\subset (\nabla f)^{-1}(0)+\frac{2}{\kappa_f} \Vert y \Vert \ball=S+\frac{2}{\kappa_f} \Vert y \Vert \ball,
$$
where $k_f$ is defined in (\ref{quad}).
Indeed, let $x\in (\nabla f)^{-1}(y)$ and $x^*={\rm proj}_S(x)$. Using (\ref{quad}) and the convexity of $f$, one has
$$
\frac{\kappa_f}{2}\Vert x-x^*\Vert^2\le f(x)-f(x^*)\le \langle \nabla f(x),x-x^* \rangle\le \Vert y \Vert \Vert x-x^*\Vert
$$
and the conclusion follows.\qed 
\end{proof}}
\section{Main result}

In this section, we study the convergence analysis of Tikhonov regularization in solving (\ref{main}). Traditionally,  we only know that the solution $x_{\varepsilon}$ of the regularized  problem  converges to the least-norm solution of (\ref{main}). However, can we estimate how fast {is} this convergence?

When $A$ is a maximal monotone operator with $0\in {\rm dom}({A}^{-1})$ then ${A}^{-1}$ is  {truncately $R$-continuous} at $0$ with some continuity modulus function $\rho$ (Proposition \ref{rela}).  If   {the function} $\rho$ can be estimated then we can answer the question above.

Since $A^{-1}$ is also maximal monotone operator,  $S:=A^{-1}(0)$ is a closed  convex set \cite{Brezis} and thus $\tilde{x}:={\rm proj}_S(0)$ is well-defined.  We have the following results.

\begin{theorem}\label{th1}
Suppose that $A:  {H}\rightrightarrows {H}$ is a maximal monotone operator, $0\in {\rm rge}({A})$ and $A^{-1}$ is {truncately $R$-continuous} at $0$ with modulus function $\rho$. Let  $\varepsilon>0$ small enough (e. g., $\varepsilon\le \frac{\sigma}{\Vert a \Vert}$) and $x_\varepsilon:= (A+\varepsilon Id)^{-1}(0)$, then 
\beq\label{estx}
\Vert x_\varepsilon \Vert \le \Vert \tilde{x} \Vert,
\eeq
and
\beq\label{conver}
 d(x_\varepsilon,S) \le\rho(\varepsilon \Vert \tilde{x} \Vert)=\rho(\varepsilon a)
\eeq
 where $\sigma>0$, $\rho(\cdot)$ are defined in (\ref{rho}), $a:= \Vert \tilde{x} \Vert$ where $ \tilde{x}$ is the least norm element of $S$. In particular the sequence $(\Vert x_\varepsilon \Vert )$ is increasing as $\varepsilon$ decreases.
Furthermore
\beq
\lim_{\varepsilon\to 0} x_\varepsilon =\tilde{x}
\eeq
with the rate 
\beq
\Vert x_\varepsilon-\tilde{x}\Vert\le \rho(\varepsilon a)+\sqrt{\rho^2(\varepsilon a)+2\rho(\varepsilon a)a}.
\eeq
\end{theorem}
\begin{proof}
Note that $x_\varepsilon$ is well-defined since $A+\varepsilon Id$ is strongly monotone.
 We have 
$$
0\in A\tilde{x}\;\;{\rm and}\;\;-\varepsilon x_\varepsilon \in Ax_\epsilon.
$$
The monotonicity of $A$  implies that
\beq\label{ep1}
\varepsilon \langle x_\varepsilon, x_\varepsilon-\tilde{x}\rangle\le 0,
\eeq
and (\ref{estx}) follows. For $\delta>\varepsilon>0$,  replacing $A$ by $A+\varepsilon I$ in the procedure above, we have $\Vert x_\delta\Vert\le \Vert x_\varepsilon\Vert$. Thus the sequence $(\Vert x_\varepsilon \Vert )$ is increasing as $\varepsilon$ decreases and bounded above by $a$.
In addition, we have 
\beq
x_\varepsilon \in A^{-1}(-\varepsilon x_\varepsilon)\subset A^{-1}(0)+\rho(\varepsilon \Vert x_\varepsilon \Vert)\ball\subset  S+\rho(\varepsilon a)\ball.
\eeq
Thus 
 $$d(x_\varepsilon,S) \le\rho(\varepsilon a).$$
Let $y_\varepsilon:={\rm proj}_S(x_\varepsilon)$ then $\Vert x_\varepsilon-y_\varepsilon\Vert \le \rho(\varepsilon a).$ We have 
\baqn
\Vert y_\varepsilon-\tilde{x}\Vert^2&\le& \Vert y_\varepsilon \Vert^2-\Vert\tilde{x}\Vert^2 \;({\rm since\;} \tilde{x}:={\rm proj}_S(0) \;{\rm and}\; y_\varepsilon\in S)\\
&\le& (\Vert  x_\varepsilon \Vert+ \Vert  x_\varepsilon-y_\varepsilon \Vert)^2-a^2\\
&\le& \rho^2(\varepsilon a)+2\rho(\varepsilon a)a.
\eaqn
Thus
$$
\Vert x_\varepsilon-\tilde{x}\Vert\le \Vert x_\varepsilon-y_\varepsilon\Vert+\Vert y_\varepsilon-\tilde{x}\Vert\le \rho(\varepsilon a)+\sqrt{\rho^2(\varepsilon a)+2\rho(\varepsilon a)a}.
$$
\qed
\end{proof}
\begin{remark}
\textit{ The continuity modulus function $\rho$  of $A^{-1}$ at zero plays an essential role in the convergence analysis. The notion of {$R$-continuity} is  also important in DC programming when $A$ is not necessarily monotone as we see in the next section.}
\end{remark}
\section{Application}
\subsection {Convex optimization}
Let us consider the minimization problem of a proper lower semicontinuous convex function $f: H\to \mathbb{R}\cup \{+\infty\}$. Suppose that the  solution set  $S:=\partial f^{-1}(0)$ is non-empty and the minimum value $f^*:=f(x^*)$ for some $x^*\in S$. Let $g=f+\frac{\varepsilon}{2} \Vert \cdot \Vert^2$ for some small  $\varepsilon>0$.  Then $g$ is strongly convex and has a unique minimizer $x_\varepsilon$ which converges to the least-norm solution $\tilde{x}\in S$ of the original problem as $\varepsilon \to 0$ \cite{Tikhonov}.  Obviously solving the strongly convex problem is easier and faster than solving the convex problem. However, it is important to know the estimations of $d(x_\varepsilon, S), \Vert x_\varepsilon-\tilde{x}\Vert$ and $ f(x_\varepsilon)-f^*$ to decide for which function $f$, it is really better to use 
Tikhonov regularization and how to choose a suitable  $\varepsilon$ which should not be too small. 
\begin{theorem}\label{th2}
Suppose that $\partial f^{-1}$ is {truncately $R$-continuous} at zero with  modulus function $\rho$. For $\varepsilon>0$ small enough, let $x_\varepsilon:=(\partial f+\varepsilon I)^{-1}(0)=\partial g^{-1}(0).$ Then we have 
\beq
 f(x_\varepsilon)-f^*\le  \Vert\tilde{x} \Vert \rho(\varepsilon \Vert\tilde{x}\Vert)\varepsilon.
\eeq
\end{theorem}
\begin{proof}
Using Theorem 1, there exists some $x^*\in S$ such that
$$
\Vert x_\varepsilon-x^*\Vert\le \rho(\varepsilon \Vert\tilde{x}\Vert).
$$
 Since $x_\varepsilon$ is the unique minimizer of $g=f+\frac{\varepsilon}{2} \Vert \cdot \Vert^2$, we deduce that
$$
f^*+\frac{\varepsilon}{2} \Vert \tilde{x}\Vert^2=f(\tilde{x})+\frac{\varepsilon}{2} \Vert \tilde{x}\Vert^2\ge  f(x_\varepsilon)+\frac{\varepsilon}{2} \Vert x_\varepsilon\Vert^2.
$$
Therefore
\baqn
 f(x_\varepsilon)-f^*&\le&\frac{\varepsilon}{2} \Vert \tilde{x}\Vert^2-\frac{\varepsilon}{2} \Vert x_\varepsilon\Vert^2\le \varepsilon \Vert \tilde{x}\Vert (\Vert \tilde{x}\Vert-\Vert {x}_\varepsilon\Vert)\\
 &\le&\varepsilon \Vert \tilde{x}\Vert (\Vert {x}^*\Vert-\Vert {x}_\varepsilon\Vert)\le\varepsilon \Vert \tilde{x}\Vert \Vert {x}^*- {x}_\varepsilon\Vert\\
  &\le&\Vert \tilde{x}\Vert \rho(\varepsilon \Vert\tilde{x}\Vert)\varepsilon.
\eaqn
\qed
\end{proof}
The following result follows by Theorem \ref{th1} and \ref{th2}.
\begin{corollary}
Suppose that $\partial f^{-1}$ is {truncately $R$-Lipschitz} at zero with modulus $c>0$. For $\varepsilon>0$ small enough, let $x_\varepsilon:=(\partial f+\varepsilon I)^{-1}(0)=\partial g^{-1}(0).$ Then we have
\baq\label{lips}
\Vert x_\varepsilon-{x}^*\Vert\le c \Vert \tilde{x}\Vert \varepsilon,\;\; \Vert x_\varepsilon-\tilde{x}\Vert\le \Vert \tilde{x}\Vert(c\varepsilon  +\sqrt{c^2\varepsilon^2+2 c\varepsilon}) 
\eaq
for some $x^*\in S$ and 
\beq
 f(x_\varepsilon)-f^*\le c \Vert\tilde{x}\Vert^2\varepsilon^2.
\eeq
\end{corollary}
\begin{example}
Let us consider the minimization problem of a quadratic function $f(x)=\frac{1}{2}\langle {B}x, x \rangle-Cx$, where $B\in \mathbb{R}^{n\times n}$ is a symmetric, positive semi-definite matrix, $C$ is some vector  in  $\mathbb{R}^n$. The solution set is given by
$$
S=\{x\in \mathbb{R}^n: Bx-C=0\}.
$$
Suppose that $S\neq \emptyset$. Let  $A(x):=Bx-C, \forall x\in \mathbb{R}^n$ and $x_\varepsilon=(B+\varepsilon Id)^{-1}C$. Then  $A^{-1}$ is {$R$-Lipschitz} continuous at zero with modulus $\frac{1}{k}$, where $k$ is the least positive eigenvalue of $B$ (Proposition \ref{lipm}).
 Thus  $(x_\varepsilon)$ converges to the least norm element $\tilde{x}$ of $S$ and 
 $$
 d(x_\varepsilon,S)\le \frac{ \Vert \tilde{x} \Vert \varepsilon}{k},\;\; f(x_\varepsilon)-f^*\le \frac{ \Vert \tilde{x}\Vert^2\varepsilon^2 }{k}.
 $$

 More specifically, we consider in $ \mathbb{R}^3$ with 
$$
B=\left( \begin{array}{ccc}
22 &\;\; 46 &\;\; 68\\ 
46 & \;\; 97 &\;\; 143 \\ 
68 & \;\; 143 &\;\; 211
\end{array} \right), \;\;C= (318 \;\;669 \;\;987)^T.
$$
Then $B$ is singular and $330$-Lipschitz continuous.
Let $\varepsilon=10^{-5}$ we can compute directly   that $$x_\varepsilon= (B+\varepsilon I)^{-1}(C)\approx(1\; 2\; 3)^T+10^{-5}(-0.2218 \;\;\;0.167\;\;-0.0557)^T.$$  It is easy to check that $\tilde{x}=(1 \; 2\; 3)^T$ is the least-norm solution  of $S$ and
$$
f(x_\varepsilon)-f^*=f(x_\varepsilon)-f(\tilde{x})\approx2.7285 \times10^{-12}.
$$
 If we use  Nesterov's Accelerated Gradient method  \cite{Nesterov}, with the initial point $y_0= (5 \;2\;3)^T$ after $k=10^4$ iterations we can find $x_k\approx(2.3334\;\;3.3333\;1.6667)^T$ with 
$$
f(x_k)-f^*\approx 9.9135\times 10^{-10}.
$$
\end{example}

\begin{example}
Next we want to minimize  a  $L$-smooth convex function $f: H\to \mathbb{R}$, i. e.,  $f$ is convex differentiable and $f'$ is $L$-Lipschitz continuous.
Nesterov's  method  \cite{Nesterov}  can construct a sequence $(x_k)$ such that 
\beq\label{esn}
f(x_k)-f^*\le \frac{2L\Vert x_0-x^*\Vert^2}{(k+2)^2},
\eeq
where $x^*$ is a minimizer of $f$ and $f^*=f(x^*).$
Let $g=f+\frac{\varepsilon}{2}\Vert \cdot \Vert^2$ for some small $\varepsilon>0$ then  $g'=f'+\varepsilon I$ is $\varepsilon$-strongly monotone and $(L+\varepsilon)$-Lipschitz continuous. Another version of Nesterov's  method for strongly convex and smooth functions \cite[p. 94]{Nesterovb} allows us to  construct a sequence $(y_k)$ such that 
\beq\label{strong}
\frac{\varepsilon}{2}\Vert y_k- x_\varepsilon \Vert^2\le g(y_k)-g(x_\varepsilon)\le \frac{L+2\varepsilon}{2}\Vert y_0-x_\varepsilon\Vert^2e^{-k\sqrt{\varepsilon/(L+\varepsilon)}}=:S_k,
\eeq
where  $x_\varepsilon=(f'+\varepsilon I)^{-1}(0)$. Therefore, using Theorem 2, we have 
\baqn
 f(y_k)-f^*&=&g(y_k)-g(x_\varepsilon) - \frac{\varepsilon}{2}(\Vert y_k\Vert^2-\Vert x_\varepsilon\Vert^2)+f(x_\varepsilon)-f^*\\
&\le& g(y_k)-g(x_\varepsilon) +\frac{\varepsilon}{2} (\Vert y_k-x_\varepsilon \Vert^2)+\varepsilon \Vert \tilde{x}\Vert\Vert y_k-x_\varepsilon \Vert+\varepsilon \Vert \tilde{x}\Vert\rho (\varepsilon \Vert \tilde{x}\Vert)\\
&\le& W_k+R(\varepsilon),
\eaqn
where 
$$
W_k:=g(y_k)-g(x_\varepsilon) +\frac{\varepsilon}{2} (\Vert y_k-x_\varepsilon \Vert^2)+\varepsilon \Vert \tilde{x}\Vert\Vert y_k-x_\varepsilon \Vert \;\;{\rm and}\;\;R(\varepsilon):=\varepsilon \Vert \tilde{x}\Vert\rho (\varepsilon \Vert \tilde{x}\Vert).
$$
While $W_k$ tends to zero very fast, the remain term $R(\varepsilon)$ does not depends on $k$. Thus it is easy to see that if $\rho(s)=cs^\alpha$ for some $\alpha>0$ close to zero and $c>0$, it is better to apply directly the  Nesterov's method without using the regularization. However, for big problems when $L$ is very large (e.g., $L\ge 10^{10}$)  if $\rho(s)=cs^\alpha$ with $\alpha\ge 1$ (e. g., in large scale quadratic programming with $\alpha=1$) and an acceptable $\varepsilon$ (e.g., $\varepsilon=10^{-5}$), 
using the regularization can reduce remarkably the iterations. Note that $L$ is  dominated by the exponential term in (\ref{strong}) and $\Vert x_\varepsilon \Vert\le \Vert \tilde{x}\Vert\le \Vert x^* \Vert$.
\end{example}
\subsection{DC programming} 
DC (difference of convex) optimization has the following form 
\beq
\min_{x\in H}f(x)=\min_{x\in H}(g(x)-h(x)),
\eeq 
where $g, h: H\to \mathbb{R}\cup \{+\infty\}$ are proper, convex, lower semicontinuous function. It is an important class of non-convex optimization and covers wide range of applications (see, e. g., \cite{Hare,Horst,Leth,main,Pham} and the references therein). One usually wants to find a critical point $x^*$ of $f$, i. e., $0\in \partial^F f(x^*)\subset \partial g(x^*)-\partial h(x^*)$ whenever $h$ is continuous at $x^*$, where $\partial^F$ denotes the Fr\'echet subdifferential and $\partial$ denotes the subdifferential in the sense of convex analysis. 
Suppose that  $ h$ is differentiable. Then $\partial^F f(x^*)=\partial g(x^*)-\nabla h(x^*)$ (see, e. g., \cite{Leth}). Applying the  \textit{forward-backward}  splitting algorithm \cite{Passty}, we can generate a sequence $(x_k)$ as follows
\beq\label{xk}
x_0\in H, x_{k+1}= J_{\gamma\partial g}L_{-\gamma\nabla h}(x_k),
\eeq
where $J_{\gamma\partial g}$ is the resolvent of $\gamma\partial g$ and $L_{-\gamma\nabla h}:=Id+\gamma \nabla h.$
\begin{theorem}
 Suppose that $f$ is bounded below, $\nabla h$ is single-valued uniformly continuous, $S:=(\partial g-\nabla h)^{-1}(0)\neq \emptyset$ and $(\partial g-\nabla h)^{-1}$ is $R$-continuous at zero with its  modulus function $\rho$.  Let $(x_k)$ be the sequence generated by the  \textit{forward-backward}  splitting algorithm (\ref{xk}). Then 
$$
\lim_{k\to +\infty}{\rm d}(x_{k},S)=0.
$$
\end{theorem}
\begin{proof}
We have
\baq\nonumber
&&x_{k+1}=(I+\gamma\partial g)^{-1}(x_k+\gamma \nabla h(x_k))\\\nonumber
 &\Leftrightarrow& x_k+\gamma \nabla h(x_k) \in x_{k+1}+\gamma\partial g(x_{k+1})\\
& \Leftrightarrow& \nabla h(x_k)-\frac{x_{k+1}-x_k}{\gamma}\in \partial g(x_{k+1}).
\label{fa1}
\eaq
On the other hand, since $g$ and $h$ are convex, we have 
\beq\label{fa2}
\langle \nabla h(x_k),x_{k+1}-x_k\rangle\le h(x_{k+1})-h(x_k)
\eeq
and
\beq\label{fa3}
\langle \nabla h(x_k)-\frac{x_{k+1}-x_k}{\gamma},x_{k+1}-x_k\rangle\ge g(x_{k+1})-g(x_k).
\eeq
From   (\ref{fa2}) and  (\ref{fa3}), one has 
\beq
\gamma(f(x_{k+1})-f(x_k))\le -  \Vert x_{k+1}-x_k\Vert^2.
\eeq
Consequently,
$$
\sum_{k=0}^{+\infty}\Vert x_{k+1}-x_k\Vert^2\le \gamma(f(x_0)- f^*)<+\infty,
$$
where $f^*:=\inf_{x\in H}f(x)$. In particular 
\beq
\lim_{k\to+\infty} \Vert x_{k+1}-x_k\Vert=0.
\eeq
From (\ref{fa1}), one has 
$$
\nabla h(x_k)-\nabla h(x_{k+1})-\frac{x_{k+1}-x_k}{\gamma}\in \partial g(x_{k+1})-\nabla h(x_{k+1}).
$$
Let $r_k:=\nabla h(x_k)-\nabla h(x_{k+1})-\frac{x_{k+1}-x_k}{\gamma}$. Then $r_k\to 0$ as $k\to+\infty$ and 
$$
x_{k+1}\in  (\partial g - \nabla h)^{-1}(r_k)\subset (\partial g - \nabla h)^{-1}(0) +\rho(\Vert r_k \Vert).
$$
Thus 
$$
\lim_{k\to +\infty}{ d}(x_{k+1},S)=0. 
$$
\qed
\end{proof}
\begin{remark}
$i)$   We  recall the well-known DCA (see, e. g. \cite{Pham})
 as follows
$$y_k\in \partial h(x_k), \;x_{k+1}\in \partial g^*(y_k)={\rm argmin} \{g(x)-[h(x_k) +\langle y_k,x-x_k \rangle] \}.$$ 
It is easy to see that DCA applied to the new couple $(\gamma g+\frac{1}{2}\Vert \cdot \Vert^2, \gamma h+\frac{1}{2}\Vert \cdot \Vert^2)$ is  the \textit{forward-backward}  splitting algorithm.  We can see that the form of \textit{forward-backward}  splitting algorithm and its convergence analysis are simpler  since we don't need  the estimations through the conjugate functions. 

ii) Theorem 3 provides a new  condition such that  $d(x_k,S)\to 0$ as $k\to +\infty$. For example, if $g=g_1+I_K$ where $g_1$ is a differentiable function and $K$ is a  convex compact set (particularly in quadratic programming over convex compact sets) then $(\partial g-\nabla h)^{-1}$ is  {$R$-continuous} at zero.
\end{remark}
\begin{corollary}\label{exam}
 Suppose that $f$ is bounded below, $\nabla h$ is single-valued uniformly continuous, $S:=(\partial g-\nabla h)^{-1}(0)\neq \emptyset$, $g=g_1+I_K$ where $g_1: H\to \mathbb{R}$ is a convex differentiable function with full domain and $K$ is a convex compact set.  Let $(x_k)$ be the sequence generated by the  \textit{forward-backward}  splitting algorithm (\ref{xk}). Then 
$$
\lim_{k\to +\infty}{\rm d}(x_{k},S)=0.
$$
\end{corollary}
\begin{proof}
It remains to prove that $\mathcal{A}:=(\partial g-\nabla h)^{-1}$ is {$R$-continuous} at zero. It is easy to see $\mathcal{A}$ has closed graph and compact range. Using  Proposition \ref{compact}, we obtain the conclusion. \qed
\end{proof}
\section{Conclusion}\label{sec13}
We believe that Tikhonov regularization has been widely used in convex optimization. The paper provides a rigorous convergence analysis with explicit estimations using the notion of   {$R$-continuity}. Consequently, it may help us to decide when we should use  Tikhonov regularization instead of the direct method. The  {$R$-continuity} of a set-valued mapping  also play an important role in DC programming. \\

%\noindent $\mathbf{Acknowledgement.}$
% The author would like to show his gratitude to the handling Editor and the Reviewer for  their careful reading and  raising many interesting  questions, which helps us to improve the paper significantly.
%\bmhead{Data Availability}
%The author confirms that all data generated or analyzed during this study are included in this manuscript. 
%\bmhead{Acknowledgments}
%
%Acknowledgments are not compulsory. Where included they should be brief. Grant or contribution numbers may be acknowledged.
%
%Please refer to Journal-level guidance for any specific requirements.

%%=============================================%%


\end{document}